\documentclass[11pt,a4paper]{article}
\usepackage{amsmath}
\usepackage{amssymb}
\usepackage{amsthm}
\usepackage{amsfonts}
\usepackage{enumerate}
\usepackage{pstricks}
\usepackage{pst-plot}
\usepackage{pst-poly}
\usepackage{url}
\usepackage{booktabs}
\usepackage{verbatim}
\usepackage{hyperref}
\usepackage{breakurl,url}
\usepackage[margin=1.3in]{geometry}
\def\tluste#1{\protect{\textrm{\boldmath $#1$}}}

\newcommand{\mace}[1]{{{#1}}}
\newcommand{\mna}[1]{{\mathcal{#1}}}

\newcommand{\omace}[1]{\mbox{$\overline{\mace{#1}}$}} 
\newcommand{\umace}[1]{\mbox{$\underline{\mace{#1}}$}} 
\newcommand{\imace}[1]{\mbox{$\tluste{#1}$}}

\def\Mid#1{#1^c}
\def\Rad#1{#1^\Delta}
\def\comp#1{{\langle{#1}\rangle}}

\newcommand{\onum}[1]{\mbox{$\overline{{#1}}$}} 
\newcommand{\unum}[1]{\mbox{$\underline{{#1}}$}}

\newcommand{\inum}[1]{\mbox{$\tluste{#1}$}} 
\newcommand{\R}[0]{{\mathbb{R}}}

\newcommand{\IR}[0]{{\mathbb{IR}}}

\def\eps{{\varepsilon}}
\newcommand{\mmid}[0]{;\,}		

\newcommand{\seznam}[1]{{\{1, \ldots, {#1}\}}}

\def\clqq{``}
\def\crqq{''}
\def\quo#1{\clqq{}#1\crqq{}}  

\newcommand{\st}[0]{{\ \ \mbox{subject to}\ \ }}
\DeclareMathOperator{\sgn}{sgn}	

\def\nref#1{\mbox{(\ref{#1})}}

\newtheorem{proposition}{Proposition}
\newtheorem{lemma}{Lemma}
\newtheorem{corollary}{Corollary}

\theoremstyle{definition}

\newtheorem{example}{Example}

\begin{document}

\title{Stability of the linear complementarity problem properties under interval uncertainty}

\author{
  Milan Hlad\'{i}k\footnote{
Charles University, Faculty  of  Mathematics  and  Physics,
Department of Applied Mathematics, 
Malostransk\'e n\'am.~25, 11800, Prague, Czech Republic, 
e-mail: \texttt{milan.hladik@matfyz.cz}
}}

\date{\today}
\maketitle

\begin{abstract}
We consider the linear complementarity problem with uncertain data modeled by intervals, representing the range of possible values. 
Many properties of the linear complementarity problem (such as solvability, uniqueness, convexity, finite  number of solutions etc.) are reflected by the properties of the constraint matrix. In order that the problem has desired properties even in the uncertain environment, we have to be able to check them for all possible realizations of interval data. This leads us to the robust properties of interval matrices. In particular, we will discuss $S$-matrix, $Z$-matrix, copositivity, semimonotonicity, column sufficiency, principal nondegeneracy, $R_0$-matrix and $R$-matrix. We characterize the robust properties and also suggest efficiently recognizable subclasses.
\end{abstract}


\section{Introduction}

\paragraph{Linear complementarity problem.}
The linear complementarity problem (LCP) appears in many optimization and operations research models such as quadratic programming, bimatrix games, or equilibria in specific economies. Its mathematical formulation reads
\begin{align}
\label{lcp1}
y&=Az+q,\ \ y,z\geq0,\\
\label{lcp2}
y^Tz&=0,
\end{align}
where $A\in\R^{n\times n}$ and $q\in\R^n$. Condition \nref{lcp1} is linear, but the (nonlinear) complementarity condition \nref{lcp2} makes the problem hard.
The LCP is called \emph{feasible} if \nref{lcp1} is feasible, and is called \emph{solvable} if \nref{lcp1}--\nref{lcp2} is feasible. Basic properties and algorithms for LCP are described, e.g., in the books \cite{CotPan2009,Mur1997}.

\paragraph{Interval uncertainty.}
Properties of the solution set of LCP relate with properties of matrix~$A$. In this paper, we study properties of $A$ when its entries are not precisely known, but we have interval ranges covering the exact values. Formally, an interval matrix is a set
$$
\imace{A}:=\{A\in\R^{m\times n}\mmid \umace{A}\leq A\leq \omace{A}\},
$$
where $ \umace{A},\omace{A}\in\R^{m\times n}$, $\umace{A}\leq\omace{A}$, are given matrices and the inequality is understood entrywise. 
The corresponding midpoint and radius matrices are defined as
$$
\Mid{A}:=\frac{1}{2}(\umace{A}+\omace{A}),\quad
\Rad{A}:=\frac{1}{2}(\omace{A}-\umace{A}).
$$

The LCP with interval uncertainties was addressed in \cite{AleScha2003,MaXu2009}, among others. They investigated the problem of enclosing the solution set of all possible realizations of interval data. Our goal is different, we focus on the interval matrix properties related to the LCP.

\paragraph{Problem statement.}
Throughout this paper we consider a class of the LCP problems with $A\in\imace{A}$, where $\imace{A}$ is a given interval matrix.
Let $\mna{P}$ be a matrix property. We say that $\mna{P}$ holds \emph{strongly} for $\imace{A}$ if it holds for each $A\in\imace{A}$.

Our aim is to characterize strong versions of several fundamental matrix classes appearing in the context of the LCP.  If property $\mna{P}$ holds strongly for an interval matrix $\imace{A}$, then we are sure that $\mna{P}$ is provably valid whatever are the true values of the uncertain entries. Therefore, the property holds in a robust sense for the LCP problem.

\paragraph{Notation.}
Given a matrix $A\in\R^{n\times n}$ and index sets $I,J\subseteq\{1,\dots,n\}$, $A_{I,J}$ denotes the restriction of $A$ to the rows indexed by $I$ and the columns indexed by~$J$. Similarly $x_I$ denotes the restriction of a vector $x$ to the entries indexed by~$I$. The identity matrix of size $n$ is denoted by $I_n$, and the spectral radius of a matrix $A$ by $\rho(A)$. The symbol $D_s$ stands for the diagonal matrix with entries $s_1,\dots,s_n$ and $e=(1,\dots,1)^T$ for the vector of ones. The relation $x\gneqq y$ between vectors $x,y$ is defined as $x\geq y$ and $x\not= y$.

\section{Particular matrix classes}

In the following sections, we consider important classes of matrices appearing in the context of the linear complementarity problem. We characterize their strong counterparts when entries are interval valued. Other matrix properties were discussed, e.g., in \cite{GarAdm2016a,Hla2017da,Hla2019b,HorHla2017a,KreLak1998}.

In particular, we leave aside several kinds of interval matrices that were already studied: $M$-matrices, $P$-matrices and positive (semi-)definite matrices. We review the basic definitions and properties, which we will need later on.

A matrix $A\in\R^{n\times n}$ is an $M$-matrix if $A=s I_n-N$ for some $N\geq0$ such that $s>\rho(N)$. There are many other equivalent conditions known \cite{Plem1977,HorJoh1991}, among which we will use that one stating that $A$ is an $M$-matrix if an only if it is a $Z$-matrix and $A^{-1}\geq0$. By \cite{BarNud1974}, an interval matrix $\imace{A}$ is strongly an $M$-matrix if and only if $\umace{A}$ is an $M$-matrix and $\omace{A}_{ij}\leq0$ for all $i\not=j$.

Analogously, a matrix $A\in\R^{n\times n}$ is an $M_0$-matrix if $A=s I_n-N$ for some $N\geq0$ such that $s\geq\rho(N)$.
Equivalently, $A$ is a matrix with nonpositive off-diagonal entries and nonnegative real eigenvalues \cite{FiePta1962,Hog2007}.

$H$-matrices closely relate to $M$-matrices. A matrix $A\in\R^{n\times n}$ is an $H$-matrix if the so called comparison matrix $\langle A\rangle$ is an $M$-matrix, where $\langle A\rangle_{ii}=|a_{ii}|$ and  $\langle A\rangle_{ij}=-|a_{ij}|$ for $i\not=j$.
By \cite{Neu1984,Neu1990}, an interval matrix $\imace{A}$ is strongly an $H$-matrix if and only if $\langle\imace{A}\rangle$ is an $M$-matrix, where $\langle \imace{A}\rangle_{ii}=\min\{|a|\mmid a\in\inum{a}_{ii}\}$ and  $\langle\imace{A}\rangle_{ij}=-\max\{|\unum{a}_{ij}|,|\onum{a}_{ij}|\}$ for $i\not=j$.

Positive definite and positive semidefinite interval matrices were studied, e.g., in \cite{Hla2018a,Roh1994b,Roh2012b}. An interval matrix $\imace{A}$ is strongly positive semidefinite if and only if the matrix $\Mid{A}-D_s\Rad{A}D_s\in\imace{A}$ is positive semidefinite for each $s\in\{\pm1\}^n$. There are some sufficient conditions known, but the problem of checking strong positive semidefiniteness is co-NP-hard in general \cite{KreLak1998}. Similar results hold for positive definiteness.

A matrix $A\in\R^{n\times n}$ is a $P$-matrix if all principal minors of $A$ are positive. $P$-matrix property of interval matrices was addressed, e.g., in \cite{BiaGar1984,Hla2017b}.

We start with two classes that are simple to characterize both in the real and interval case, and then we discuss the computationally harder classes.

\subsection{$S$-matrix}

A matrix $A\in\R^{n\times n}$ is called an \emph{$S$-matrix} if there is $x>0$ such that $Ax>0$. The significance of this class is that the LCP is feasible for each $q\in\R^n$ if and only if $A$ is an S-matrix.

Strong $S$-matrix property of an interval matrix $\imace{A}\in\IR^{n\times n}$ is easy to characterize.

\begin{proposition}
$\imace{A}$ is strongly an $S$-matrix if and only if system $\umace{A}x>0$, $x>0$ is feasible.
\end{proposition}

\begin{proof}
If $\umace{A}x>0$, $x>0$ has a solution $x^*$, then $Ax^*\geq\umace{A}x^*>0$ for each $A\in\imace{A}$. Therefore, every $A\in\imace{A}$ is an $S$-matrix.
\end{proof}

\subsection{$Z$-matrix}

A matrix $A\in\R^{n\times n}$ is called a \emph{$Z$-matrix} if $a_{ij}\leq0$ for each $i\not=j$.
$Z$-matrices emerge in the context of Lemke's complementary pivot algorithm, because it processes any LCP with a $Z$-matrix.

It is easy to see that the strong $Z$-matrix property reduces to $Z$-matrix property of the upper bound matrix $\omace{A}$.

\begin{proposition}
$\imace{A}$ is strongly a $Z$-matrix if and only if $\omace{A}$ is a $Z$-matrix.
\end{proposition}

\subsection{Copositive matrix}

A matrix $A\in\R^{n\times n}$ is called \emph{copositive} if $x^TAx\geq0$ for each $x\geq0$. It is \emph{strictly copositive} if $x^TAx>0$ for each $x\gneqq0$. 
A copositive matrix ensures that the complementary pivot algorithm for solving the LCP works. 
A strictly copositive matrix in addition implies that the LCP has a solution for each $q\in\R^n$. 
Checking whether $A$ is copositive is a co-NP-hard problem \cite{MurKab1987}.

From the definition immediately follows:

\begin{proposition}\label{propCopos}
$\imace{A}$ is strongly (strictly) copositive if and only if $\umace{A}$ is (strictly) copositive.
\end{proposition}

Matrix $A$ is (strictly) copositive if and only if its symmetric counterpart $\frac{1}{2}(A+A^T)$ is (strictly) copositive. That is why we can without loss of generality focus on symmetric matrices. In particular, we assume that $\Mid{A}$ and $\Rad{A}$ are symmetric.

Since checking copositivity is co-NP-hard, it is desirable to inspect some polynomially solvable classes of problems.

\begin{proposition}
Let $\Mid{A}$ be an $M$-matrix. Then
\begin{enumerate}[(1)]
\item
$\imace{A}$ is strongly copositive if and only if $\umace{A}$ is an $M_0$-matrix;
\item
$\imace{A}$ is strongly strictly copositive if and only if $\umace{A}$ is an $M$-matrix.
\end{enumerate}
\end{proposition}

\begin{proof}
(1) \quo{If.}
If $\umace{A}$ is an $M_0$-matrix, then it is positive semidefinite \cite{CotPan2009} and so it is copositive. By Proposition~\ref{propCopos}, strong copositivity of $\imace{A}$ follows.

\quo{Only if.}
Suppose to the contrary that $\umace{A}$ is not an $M_0$-matrix. Then we can write $A=sI_n-N$, where $N\geq0$ and $\rho(N)>s$. For the corresponding Perron vector $x\gneqq0$ we have $Nx=\rho(N)x\gneqq sx$, from which $Ax\lneqq 0$. If $x_i=0$, then $(Nx)_i=0$ and so $(Ax)_i=0$. Similarly, if $x_i>0$, then $(Nx)_i>sx_i$ and so $(Ax)_i<0$. Hence $Ax$ and $x$ have the same nonzero entries, whence $x^TAx<0$; a contradiction.

(2) For strict copositivity we proceed analogously.
\end{proof}

\begin{corollary}
Let $\Mid{A}=I_n$. Then 
\begin{enumerate}[(1)]
\item
$\imace{A}$ is strongly copositive if and only if $\rho(\Rad{A})\leq 1$;
\item
$\imace{A}$ is strongly strictly copositive if and only if $\rho(\Rad{A})<1$.
\end{enumerate}
\end{corollary}

\begin{proof}
Obviously, $\Mid{A}=I_n$ is an $M$-matrix. Further, $I_n-\Rad{A}$ is an $M_0$-matrix if and only if $\rho(\Rad{A})\leq1$. Similarly for strict copositivity.
\end{proof}

\subsection{Semimonotone matrix}

A matrix $A\in\R^{n\times n}$ is called \emph{semimonotone (an $E_0$-matrix)} if the LCP has a unique solution for each $q>0$. Equivalently, for each index set $\emptyset\not=I\subseteq\{1,\dots,n\}$ the system
\begin{align}\label{sysSemimono}
A^{}_{I,I}x<0,\ \ x\geq0
\end{align}
is infeasible. 
By \cite{Tsen2000}, checking whether $A$ is semimonotone is a co-NP-hard problem.

From the definition we simply derive:

\begin{proposition}
$\imace{A}$ is strongly semimonotone if and only if $\umace{A}$ is semimonotone.
\end{proposition}

The next result shows a class of interval matrices, for which checking strong semimonotonicity can be performed effectively in polynomial time.

\begin{proposition}
Let $\Mid{A}$ be an $M_0$-matrix. Then $\imace{A}$ is strongly semimonotone if and only if $\umace{A}$ is an $M_0$-matrix.
\end{proposition}

\begin{proof}
\quo{If.}
Suppose to the contrary that there are $A\in\imace{A}$ and $I$ such that \nref{sysSemimono} has a solution~$x$. Without loss of generality assume that $x>0$; otherwise we restrict to $A_{I',I'}$, where $I'=\{i\in I\mmid x_i>0\}$. Since $\umace{A}$ is an $M_0$-matrix, also $\umace{A}_{I,I}$ is an $M_0$-matrix. That is, we can write it as $\umace{A}_{I,I}=sI_m-N$, where $N\geq0$ and $\rho(N)\leq s$. However, from \nref{sysSemimono} we have $Nx>sx$, from which $\rho(N)> s$; a contradiction.

\quo{Only if.}
Suppose to the contrary that $\umace{A}$ is not an $M_0$-matrix. That is, $\umace{A}=sI_n-N$, where $N\geq0$ and $\rho(N)> s$. Let $x\gneqq0$ be the Perron vector corresponding to~$N$, so that $Nx=\rho(N)x\geq sx$. Then $\umace{A}x=sx-Nx\leq0$. Define $I=\{i\mmid x_i>0\}$. Since $(Nx)_i>sx_i$ for each $i\in I$, we get $\umace{A}_{I,I}x_I<0$ for $x_I>0$; a contradiction.
\end{proof}

\begin{corollary}
Let $\Mid{A}=I_n$. Then $\imace{A}$ is strongly semimonotone if and only if $\rho(\Rad{A})\leq 1$.
\end{corollary}

\subsection{Principally nondegenerate matrix}

A matrix $A\in\R^{n\times n}$ is called \emph{principally nondegenerate} if all its principal minors are nonzero. 
In the context of the LCP, such matrix guarantees that the problem has finitely many solutions (including zero) for every $q\in\R^n$.

From the definition, an interval matrix $\imace{A}$ is strongly principally nondegenerate if its principal submatrices are strongly nonsingular (i.e., contain nonsingular matrices only).
Below, we state a finite reduction. By the same reasoning as in \cite{Hla2017b} we can enumerate that the condition requires $5^n$ instances. This is a high number, but justified by two facts: First, checking principally nondegeneracy of a real matrix is co-NP-hard \cite{Tsen2000}. Second, checking whether an interval matrix is strongly nonsingular matrix is co-NP-hard, too \cite{PolRoh1993}.

\begin{proposition}
$\imace{A}$ is strongly principally nondegenerate if and only if
\begin{align}\label{ineqPropNondegStr}
\det\left(D_{e-|y|}+D_{|y|}\Mid{A}D_{|z|}\right)
\det\left(D_{e-|y|}+D_{|y|}\Mid{A}D_{|z|}-D_{y}\Rad{A}D_{z}\right)>0
\end{align}
for each $y,z\in\{0,\pm1\}^m$ such that $|y|=|z|$.
\end{proposition}

\begin{proof}
By \cite{Roh2009}, $\imace{A}$ is strongly nonsingular if and only if
\begin{align*}
\det(\Mid{A})\det(\Mid{A}-D_{y}\Rad{A}D_{z})>0,\ \ \forall y,z\in\{\pm1\}^m.
\end{align*}
We need to check this condition for each principal submatrix of~$\imace{A}$. We claim it is as stated in \nref{ineqPropNondegStr}. 
Consider the permutation of rows and columns, represented by the permutation matrix~$P$, that brings the zeroes of $y,z$ into the last entries. Then the matrix $P(D_{e-|y|}+D_{|y|}\Mid{A}D_{|z|}-D_{y}\Rad{A}D_{z})P^T$ is a block diagonal matrix. The right bottom block is the identity matrix and the left top block is the principal submatrix of $\Mid{A}-D_{y}\Rad{A}D_{z}$ indexed by $I:=\{i\mmid |y_i|=1\}=\{i\mmid |z_i|=1\}$.

Similarly for $D_{e-|y|}+D_{|y|}\Mid{A}D_{|z|}$.
\end{proof}

An efficient test can be performed only for specific types of matrices.
Since principally nondegenerate matrices are closed under nonzero row or column scaling, the following can directly be extended to an interval matrix $\imace{A}$ such that $D_y\Mid{A}D_z$ is an $M$-matrix for some $y,z\in\{\pm1\}^n$.

\begin{proposition}
Let $\Mid{A}$ be an $M$-matrix. Then $\imace{A}$ is strongly principally nondegenerate if and only if it is an $H$-matrix. 
\end{proposition}

\begin{proof}
By Neumaier \cite[Prop.~4.1.7]{Neu1990}, our assumption implies that $\imace{A}$ is strongly nonsingular if and only if $\imace{A}$ is an $H$-matrix.
Principal submatrix of an $M$-matrix is again an $M$-matrix, and the same property holds for $H$-matrices. Therefore the above result applies for principal submatrices, too, from which the statement follows.
\end{proof}

Under the assumption that $\Mid{A}$ is an $M$-matrix we have that $\imace{A}$ is an $H$-matrix if and only if $\umace{A}$ is an $M$-matrix. Therefore we can equivalently test whether $\umace{A}$ is an $M$-matrix in the above proposition.

\begin{corollary}
Let $\Mid{A}=I_n$. Then $\imace{A}$ is strongly principally nondegenerate if and only if $\rho(\Rad{A})<1$.
\end{corollary}

\begin{proposition}
Let $\Mid{A}$ be positive definite. Then $\imace{A}$ is strongly principally nondegenerate if and only if it strongly positive definite. 
\end{proposition}

\begin{proof}
Since $\Mid{A}$ is positive definite, we have by \cite{Roh1994b} that $\imace{A}$ is strongly nonsingular if and only is it is strongly positive definite. 
Since positive definiteness is preserved to principal submatrices, the statement follows.
\end{proof}

Checking strong positive definiteness of $\imace{A}$ is known to be co-NP-hard, but there are various sufficient conditions known; see \cite{Roh1994b}.

\subsection{Column sufficient matrix}

A matrix $A\in\R^{n\times n}$ is \emph{column sufficient} if for each pair of disjoint index sets $I,J\subseteq\{1,\dots,n\}$, $I\cup J\not=\emptyset$, the system
\begin{align}\label{sysColSuff}
\begin{pmatrix}A^{}_{I,I}&-A^{}_{I,J}\\-A^{}_{J,I}&A^{}_{J,J}\end{pmatrix}
x\lneqq0,\ \ x>0
\end{align}
is infeasible.
Checking this condition is co-NP-hard \cite{Tsen2000}, which justifies necessity of inspecting all index sets $I,J$.
Among other properties, column sufficiency implies that for any $q\in\R^n$ the solution set of the LCP is a convex set (including the empty set).

\begin{proposition}
$\imace{A}$ is strongly column sufficient if and only if system 
\begin{align}\label{sysPropColSuffStr}
\begin{pmatrix}
 \umace{A}_{I,I}&-\omace{A}_{I,J}\\
 -\omace{A}_{J,I}&\umace{A}_{J,J}
\end{pmatrix}
x\lneqq0,\ \ x>0
\end{align}
is infeasible for each admissible $I,J$.
\end{proposition}

\begin{proof}
If $\imace{A}$ is strongly column sufficient, then \nref{sysPropColSuffStr} must be infeasible, because the matrix there comes from $\imace{A}$. Conversely, if some  $A\in\imace{A}$ is not column sufficient, then \nref{sysColSuff} has a solution $x^*$ for certain $I,J$. Since
$$
\begin{pmatrix}
 \umace{A}_{I,I}&-\omace{A}_{I,J}\\
 -\omace{A}_{J,I}&\umace{A}_{J,J}
\end{pmatrix}x^*
\leq
\begin{pmatrix}A_{I,I}&-A_{I,J}\\-A_{J,I}&A_{J,J}\end{pmatrix}x^*,
$$
we have that $x^*$ is a solution to \nref{sysPropColSuffStr}; a contradiction.
\end{proof}

The above result also suggests a reduction to finitely many (namely, $2^n$) instances.

\begin{proposition}
$\imace{A}$ is strongly column sufficient if and only if matrices of the form $A_{ss}=\Mid{A}-D_s\Rad{A}D_s$ are  column sufficient for each $s\in\{\pm1\}^n$.
\end{proposition}

\begin{proof}
If $\imace{A}$ is strongly column sufficient, then $A_{ss}$ is column sufficient since $A_{ss}\in\imace{A}$. Conversely, if $\imace{A}$ is not strongly column sufficient, then \nref{sysPropColSuffStr} has a solution. However, feasibility of  \nref{sysPropColSuffStr} implies that $A_{ss}$ is not column sufficient for $s\in\{\pm1\}^n$ defined as follows: $s_i:=1$ if $i\in I$ and $s_i:=-1$ otherwise, because
\begin{equation*}
(A_{ss})_{I,I}=\umace{A}_{I,I},\ \ 
(A_{ss})_{J,J}=\umace{A}_{J,J},\ \ 
(A_{ss})_{I,J}=\omace{A}_{I,J}.
\qedhere
\end{equation*}
\end{proof}

Below, we state a polynomially recognizable class. To this end, recall that a nonnegative matrix $A$ is called \emph{irreducible} if $(I_n+A)^{n-1}>0$, or equivalently, $P^TAP$ is block triangular (with at least two blocks) for no permutation matrix~$P$.

\begin{lemma}\label{lmmPosMaceRho}
Let $A\geq0$ be irreducible and $x\geq0$ such that $x\lneqq Ax$. Then $\rho(A)>1$.
\end{lemma}

\begin{proof}
Define $\tilde{x}:=(I_n+A)^{n-1}x>0$. Since $x\lneqq Ax$, we have
$$
(I_n+A)^{n-1}x< (I_n+A)^{n-1}Ax=A(I_n+A)^{n-1}x,
$$
from which $\tilde{x}<A\tilde{x}$. By \cite{HorJoh1985}, $\rho(A)>1$.
\end{proof}

\begin{proposition}\label{propColSuffMmat}
Let $\Mid{A}$ be an $M$-matrix and $\Rad{A}$ irreducible. Then $\imace{A}$ is strongly column sufficient if and only $\umace{A}$ is an $M_0$-matrix. 
\end{proposition}

\begin{proof}
\quo{If.}
Suppose to the contrary that \nref{sysColSuff} has a solution $x$ for some $I,J$ and $A\in\imace{A}$. Then the same solution solves $A'x\lneqq 0$, $x>0$, where $A'=\umace{A}_{I\cup J,I\cup J}$. Thus we can assume without loss of generality that $J=\emptyset$; otherwise we set $I:=I\cup J$ and  $J:=\emptyset$. 
In addition, consider the instance, for which $I$ has maximal cardinality.
Since $A'$ is an $M_0$-matrix, it can be expressed as $A'=sI_m-N$, where $N\geq0$ and $\rho(N)\leq s$. Then 
$$
(sI_m-N)x\lneqq 0,\ \  x>0,
$$
from which $Nx\gneqq sx$. 

Consider two cases:
(1) $A'$ is irreducible. Then $N$ is irreducible and by Lemma~\ref{lmmPosMaceRho}, we have $\rho(N)>s$; a contradiction.

(2) $A'$ is reducible. Define $J':=\seznam{n}\setminus I$ and the vector $\tilde{x}^*$ such that $\tilde{x}^*_{I}=x$ and $\tilde{x}^*_{J'}=\eps e$, where $\eps>0$ is sufficiently small. Consider the system
\begin{align*}
\begin{pmatrix}
 \umace{A}^{}_{I,I}&\umace{A}^{}_{I,J'}\\
 \umace{A}^{}_{J',I}&\umace{A}^{}_{J',J'}
\end{pmatrix}
\tilde{x}\lneqq0,\ \ \tilde{x}>0.
\end{align*}
Since $\umace{A}^{}_{I,J'}\leq0$, vector $\tilde{x}^*$ solves the first block of inequalities. If matrix $\umace{A}^{}_{J',I}$ contains a nonzero row, say~$i$, then $\tilde{x}^*$ solves also the $i$th inequality (since $\umace{A}^{}_{J',I}\leq0$). Nevertheless, this cannot happen since we could put $I:=I\cup\{i\}$, which contradicts maximum cardinality of~$I$. Thus $\umace{A}^{}_{J',I}=0$, which means that $\Rad{A}_{J',I}=0$ and so $\Rad{A}$ is reducible; a contradiction.

\quo{Only if.}
Suppose to the contrary that $\umace{A}$ is not an $M_0$-matrix. 
Denote $A:=\umace{A}$ and split it as follows: $A=sI_n-N$, where $N\geq0$. Since $A$ is not an $M_0$-matrix, we have $\rho(N)>s$. Thus there is $x\gneqq0$ such that $Nx=\rho(N)x$. Define $I:=\{i\mmid x_i>0\}\not=\emptyset$ and $J:=\emptyset$. Then $x_I>0$ and $N_{I,I}x_I=\rho(N)x_I>sx_I$, from which $A_{I,I}x_I=(sI_m-N_{I,I})x_I<0$. 
Therefore \nref{sysColSuff} is feasible; a contradiction.
\end{proof}

Notice that assumption that $A\geq0$ is irreducible is necessary. As a counterexample, consider 
$$
\Mid{A}=\begin{pmatrix}1&0\\0&1\end{pmatrix},\quad
\Rad{A}=\begin{pmatrix}1&1\\0&1\end{pmatrix}.
$$
Then $\umace{A}=\left(\begin{smallmatrix}0&-1\\0&0\end{smallmatrix}\right)$, which is an $M_0$-matrix, but not column sufficient. 

In Proposition~\ref{propColSuffMmat}, the assumption that $\Mid{A}$ is an $M$-matrix can be directly extended such that $D_s\Mid{A}D_s$ is an $M$-matrix for some $s\in\{\pm1\}^n$. This is because column sufficient matrix $A$ is closed under transformation $D_sAD_s$.

\begin{corollary}
Let $\Mid{A}=I_n$ and $\Rad{A}$ irreducible. Then $\imace{A}$ is strongly column sufficient if and only if $\rho(\Rad{A})\leq 1$.
\end{corollary}

\begin{proposition}
Let $\Mid{A}$ be positive semidefinite. Then $\imace{A}$ is strongly column sufficient if and only if it is strongly positive semidefinite.
\end{proposition}

\begin{proof}
\quo{If.}
This follows from the fact that positive semidefinite matrix is column sufficient:  $A$ is positive semidefinite if and only if $D_sAD_s$ is positive semidefinite, where $s\in\{\pm1\}^n$ is arbitrary. We will use it for the setting $s:=(1,\dots,1,-1,\dots,-1)^T$. Thus
\begin{align*}
\tilde{A}:=\begin{pmatrix}A^{}_{I,I}&-A^{}_{I,J}\\-A^{}_{J,I}&A^{}_{J,J}\end{pmatrix}
\end{align*}
is  positive semidefinite, too. If $x$ is a solution to \nref{sysColSuff}, then $x^T\tilde{A}x<0$; a contradiction.

\quo{Only if.}
Suppose to the contrary that $\imace{A}$ is not strongly positive semidefinite. Thus there is $s\in\{\pm1\}^n$ such that $\Mid{A}-D_s\Rad{A}D_s\in\imace{A}$ is not positive semidefinite. Let $\alpha^*$ be maximal $\alpha\geq0$ such that $\Mid{A}-\alpha D_s\Rad{A}D_s$ is positive semidefinite for each $\alpha\in[0,\alpha^*]$. Obviously, $\alpha^*<1$ and $\Mid{A}-\alpha^* D_s\Rad{A}D_s$ is singular. Then $A_{\alpha^*}:=D_s\Mid{A}D_s-\alpha^* \Rad{A}$ is singular, too. 
Let $x^*\not=0$ such that $A_{\alpha^*}x^*=0$ and let $z:=\sgn(x^*)$. 
If $x^*\gneqq0$, then $A_{\alpha=1}x^*\lneqq0$, so $x^*_I$ solves \nref{sysColSuff} with $I:=\{i\mmid x_i^*>0,\ s_i=1\}$, $J:=\{i\mmid x_i^*>0,\ s_i=-1\}$ and $A:=\Mid{A}-D_s\Rad{A}D_s$. If $x^*\lneqq0$, then we substitute $x^*:=-x^*$ and proceed as before. Consider now the remaining case. Without loss of generality write $x^*=(x^*_1,x^*_2)^T$, where $x^*_1<0$ and $x^*_2>0$; we can ignore the zero entries of $x^*$ since in \nref{sysColSuff} we restrict $I,J$ to the indices corresponding to the nonzero entries only. Denoting $\tilde{x}^*:=D_zx^*>0$, the equation $D_zA_{\alpha^*}x^*=0$ reads
\begin{align}\label{eqPfPropColSufPsd1}
(D_zD_s\Mid{A}D_sD_z-\alpha^* D_z\Rad{A}D_z)\tilde{x}^*=0.
\end{align}
Denote $\Rad{\tilde{A}}:=\frac{1}{2}(\Rad{A}+D_z\Rad{A}D_z)\geq0$. Since $D_z\Rad{A}D_z\leq\Rad{\tilde{A}}\leq\Rad{A}$, we have
$$
(D_zD_s\Mid{A}D_sD_z-\alpha^* \Rad{\tilde{A}})\tilde{x}^*\leq0,
$$
whence
$$
(D_zD_s\Mid{A}D_sD_z-\Rad{\tilde{A}})\tilde{x}^*\leq0.
$$
If at least one inequality holds strictly in this system, then $\tilde{x}^*$ is a solution to \nref{sysColSuff} with $I:=\{i\mmid s_iz_i=1\}$, $J:=\{i\mmid s_iz_i=-1\}$ and $A:=\Mid{A}-D_zD_s\Rad{\tilde{A}}D_sD_z$. 
If it is not the case, then necessarily $ \Rad{\tilde{A}}=0$, whence $D_z\Rad{A}D_z\lneqq0$. However, this is a contradiction since from \nref{eqPfPropColSufPsd1} we have
$$
(D_zD_s\Mid{A}D_sD_z)\tilde{x}^*\lneqq0,
$$
which contradicts positive semidefiniteness of $\Mid{A}$.
\end{proof}

\subsection{$R_0$-matrix}

A matrix $A\in\R^{n\times n}$ is an \emph{$R_0$-matrix} if the LCP with $q=0$ has the only solution $x=0$.
Equivalently, for each index set $\emptyset\not=I\subseteq\{1,\dots,n\}$, the system
\begin{align}\label{sysRoMat}
A^{}_{I,I}x=0,\ \ A^{}_{J,I}x\geq 0,\ \ x>0
\end{align}
is infeasible, where $J:=\{1,\dots,n\}\setminus I$.
Checking $R_0$-matrix property is co-NP-hard \cite{Tsen2000}.
If $A$ is an $R_0$-matrix, then for any $q\in\R^n$ the LCP has a bounded solution set.

\begin{proposition}\label{propRoMatStr}
$\imace{A}$ is strongly $R_0$-matrix if and only if system 
\begin{align}\label{sysPropRoMatStr}
\umace{A}^{}_{I,I}x\leq 0,\ \
\omace{A}^{}_{I,I}x\geq 0,\ \ 
\umace{A}^{}_{J,I}x\geq 0,\ \  
x>0
\end{align}
is infeasible for each admissible $I,J$.
\end{proposition}

\begin{proof}
$\imace{A}$ is not strongly an $R_0$-matrix if and only if there are $I,J$ and $A\in\imace{A}$ such that \nref{sysRoMat} is feasible. It is known \cite{Fie2006,Hla2013b} that \nref{sysRoMat} is feasible for some $A\in\imace{A}$ if and only if \nref{sysPropRoMatStr} is feasible, from which the statement follows.
\end{proof}

Despite intractability in the general case, we can formulate a polynomial time recognizable sub-class.

\begin{proposition}\label{propRoMatMidMmace}
Let $\Mid{A}$ be an $M$-matrix. Then $\imace{A}$ is strongly an $R_0$-matrix if and only if $\imace{A}$ is strongly an $H$-matrix.
\end{proposition}

\begin{proof}
\quo{If.}
This follows from nonsingularity of $H$-matrices and the fact that a principal submatrix of an H-matrix is an H-matrix.

\quo{Only if.}
Suppose to the contrary that $\imace{A}$ is not strongly an $H$-matrix. 
Thus there is $A\in\imace{A}$ that is not an $H$-matrix. Hence $\comp{A}$ is not an $M$-matrix and due to the assumption it belongs to~$\imace{A}$. 
Since $\Mid{A}$ is an $M$-matrix and from continuity reasons, $\imace{A}$ contains an $M_0$-matrix $A'$ that is not an $M$-matrix. Thus we can write $A'=sI_n-N$, where $N\geq0$ and $\rho(N)=s$. Let $x\gneqq0$ be the Perron vector corresponding to $N$, that is, $Nx=sx$, whence $A'x=0$. 
Put $I:=\{i\mmid x_i>0\}\not=\emptyset$, so $A'_{I,I}x_I=0$, $A'_{J,I}x_I=0$ and $x_I>0$. Therefore $x_I$ solves \nref{sysRoMat}.
\end{proof}

\begin{corollary}
Let $\Mid{A}=I_n$. Then $\imace{A}$ is strongly an $R_0$-matrix if and only if $\rho(\Rad{A})<1$.
\end{corollary}

\subsection{$R$-matrix}

A matrix $A\in\R^{n\times n}$ is an \emph{$R$-matrix (regular)} if for each index set $\emptyset\not=I\subseteq\{1,\dots,n\}$, the system
\begin{align}\label{sysRegMat}
A^{}_{I,I}x+et=0,\ \ A^{}_{J,I}x+et\geq 0,\ \ x>0,\ t\geq0
\end{align}
is infeasible w.r.t.\ variables $x\in R^{|I|}$ and $t\in\R$, where $J:=\{1,\dots,n\}\setminus I$. Regularity of $A$ ensures that for any $q\in\R^n$ the LCP has a solution.

\begin{proposition}
$\imace{A}$ is strongly an $R$-matrix if and only if system 
\begin{align*}
\umace{A}^{}_{I,I}x+et\leq 0,\ \
\omace{A}^{}_{I,I}x+et\geq 0,\ \ 
\umace{A}^{}_{J,I}x+et\geq 0,\ \  
x>0
\end{align*}
is infeasible for each admissible $I,J$.
\end{proposition}

\begin{proof}
Similar to the proof of Proposition~\ref{propRoMatStr}.
\end{proof}

\begin{proposition}
Let $\Mid{A}$ be $M$-matrix. Then $\imace{A}$ is strongly an $R$-matrix if and only if $\imace{A}$ is strongly an $H$-matrix.
\end{proposition}

\begin{proof}
\quo{If.}
Suppose to the contrary that \nref{sysRegMat} has a solution $x^*,t^*$ for some $I$ and $A\in\imace{A}$. If $t^*=0$, then we are done based on Proposition~\ref{propRoMatMidMmace}. So let $t^*>0$ and without loss of generality we can assume that $t^*=1$. Thus $A_{I,I}x^*=-e$, whence $\umace{A}_{I,I}x^*\leq -e$ in view of $x^*>0$. Since $\Mid{A}$ is an $M$-matrix and $\imace{A}$ is strongly an $H$-matrix, it follows that $\umace{A}=\comp{A}$ is an $M$-matrix. Then $\umace{A}_{I,I}$ is an $M$-matrix, too, and so it has a nonnegative inverse, implying $x^*\leq -\umace{A}_{I,I}^{-1}e<0$; a contradiction.

\quo{Only if.}
We proceed similarly to the proof of Proposition~\ref{propRoMatMidMmace}, where a solution $x_I$ to \nref{sysRoMat} was found. Now, put $t=0$ and $(x_I,t)$ solves \nref{sysRegMat}.
\end{proof}

\begin{corollary}
Let $\Mid{A}=I_n$. Then $\imace{A}$ is strongly an $R$-matrix if and only if $\rho(\Rad{A})<1$.
\end{corollary}

\section{Examples}

\begin{example}
Let
$$
\Mid{A}=\begin{pmatrix}0& -1& 2\\2& 0& -2\\-1& 1& 0\end{pmatrix}.
$$
This matrix is semimonotone, column sufficient, $R$-matrix and $R_{0}$-matrix, but not principally nondegenerate. Suppose that the entries of the matrix are subject to uncertainties and the displayed values are known with $10\%$ accuracy. Thus we come across to an interval matrix $\imace{A}$, whose midpoint is the displayed matrix and radius is
$$
\Rad{A}=\frac{1}{10}|\Mid{A}|
=\begin{pmatrix}0& 0.1& 0.2\\0.2& 0& 0.2\\0.1& 0.1& 0\end{pmatrix}.
$$
Based on calculations performed in \textsf{MATLAB R2017b} we checked that $\imace{A}$ is strongly semimonotone, column sufficient, $R$-matrix and $R_{0}$-matrix. That is, the properties that are valid for $\Mid{A}$ remain valid for whatever is the true value in the uncertainty set. This means, among others, that the solution set is nonempty, bounded and convex for any $A\in\imace{A}$.

In contrast, if we increase the uncertainty level to $15\%$, then none of the above properties holds strongly. Column sufficiency, $R$-matrix and $R_{0}$-matrix properties fail for the value of $I:=\{1,2,3\}$ and $J:=\emptyset$.
\end{example}

\begin{example}\label{exCqp}
Consider a convex quadratic problem
\begin{align*}
\min\ x^TCx+d^Tx \st Bx\leq b,\ x\geq0.
\end{align*}
Optimality conditions for this problem have the form of a linear complementarity problem
\begin{align*}
y=Az+q,\ \ 
y^Tz=0,\ \ 
y,z\geq0,
\end{align*}
where 
\begin{align*}
A:=\begin{pmatrix}0 & -B\\B^T & 2C\end{pmatrix},\ \ 
q:=\begin{pmatrix}b\\d\end{pmatrix},\ \ 
z:=\begin{pmatrix}u\\x\end{pmatrix}.
\end{align*}
For concreteness, consider the problem
\begin{align*}
\min\ \ &10x_1^2 +8x_1x_2+5x_2^2+x_1+x_2 \\
\st &2x_1-x_2\leq10,\ -3x_1+x_2\leq 9,\ x\geq0,
\end{align*}
so we have
$$
A=\begin{pmatrix}
  0 &  0 &  -2 &  1\\
  0 &  0 &   3 & -1\\
  2 & -3 &  20 &  8\\
 -1 &  1 &   8 & 10
\end{pmatrix}.
$$
Calculations showed that $A$ is semimonotone, column sufficient, $R$-matrix and $R_{0}$-matrix, but not principally nondegenerate. Since $q>0$, the LCP problem has a unique solution. However, even if $q$ was not positive, the other properties would imply nonemptiness, boundedness and convexity of the solution set.

\begin{table}[t]
\begin{center}
\renewcommand{\arraystretch}{1.15}
\begin{tabular}{@{}ccc@{}}
\toprule
 $\Rad{B}$ & $\Rad{C}$ & strong properties\\
\midrule
$\Rad{B}=0$ & $\Rad{C}=\frac{1}{4}|C|$ &
  semimonotone, column sufficient, $R$-matrix, $R_{0}$-matrix\\
$\Rad{B}=0$ & $\Rad{C}=\frac{1}{3}|C|$ &
  semimonotone, $R$-matrix, $R_{0}$-matrix\\
$\Rad{B}=0$ & $\Rad{C}=\frac{9}{10}|C|$ &
  semimonotone, $R$-matrix, $R_{0}$-matrix\\
$\Rad{B}=0$ & $\Rad{C}=|C|$ &
  semimonotone\\
$\Rad{B}=\frac{1}{10}|\Mid{B}|$ & $\Rad{C}=\frac{1}{10}|C|$ &
  semimonotone, column sufficient, $R$-matrix, $R_{0}$-matrix\\
$\Rad{B}=\frac{1}{10}|\Mid{B}|$ & $\Rad{C}=\frac{1}{5}|C|$ &
  semimonotone, column sufficient, $R$-matrix, $R_{0}$-matrix\\
$\Rad{B}=\frac{1}{10}|\Mid{B}|$ & $\Rad{C}=\frac{1}{2}|C|$ &
  semimonotone, $R$-matrix, $R_{0}$-matrix\\
$\Rad{B}=\frac{1}{5}|\Mid{B}|$ & $\Rad{C}=\frac{1}{5}|C|$ &
  $\emptyset$\\
\bottomrule
\end{tabular}
\caption{(Example~\ref{exCqp}) Strong properties for different uncertainty degrees.\label{tabExCqp}}
\end{center}
\end{table}
As in the previous example, we consider uncertainty in terms of maximal percentage variations of the matrix entries. Table~\ref{tabExCqp} shows the results for various degrees of uncertainty, which is represented by the radius of interval matrix $\imace{A}$, or the particular radius matrices $\Rad{B}$ and $\Rad{C}$. 
The first four settings correspond to the case, where uncertainty affects the cost matrix only, and the technological matrix remains fixed. Naturally, the higher degree of uncertainty, the less properties hold. However, even $25\%$ independent and simultaneous variations of the costs do not influence the properties we discussed.
\end{example}

\section{Conclusion}

We analysed important classes of matrices, which guarantee that the linear complementarity problem has convenient properties related to the structure of the solution set. We characterized the matrix properties in the situation, where the input coefficients have the form of compact intervals. As a consequence, we obtained robust properties for the linear complementarity problem: whatever are the true values from the interval data, we are sure that the corresponding property is satisfied.

Since many problems are hard to check even in the real case, it is desirable to investigate some easy-to-recognize cases. We proposed several such cases, but it is still a challenging problem to explore new ones.

\section{Acknowledgements}

The author was supported by the Czech Science Foundation Grant P403-18-04735S.

\bibliographystyle{abbrv}
\bibliography{int_mat_lcp}

\end{document}